\documentclass[11pt,twoside,reqno]{amsart}

\usepackage{amssymb, latexsym}
\usepackage[italian,english]{babel}
\usepackage{amsmath,amsfonts,amsthm}
\usepackage{paralist}
\usepackage{color}
\usepackage{color}

\numberwithin{equation}{section}

\def\Xint#1{\mathchoice 
  {\XXint\displaystyle\textstyle{#1}}%
  {\XXint\textstyle\scriptstyle{#1}}%
  {\XXint\scriptstyle\scriptscriptstyle{#1}}%
  {\XXint\scriptscriptstyle\scriptscriptstyle{#1}}%
  \!\int} 
\def\XXint#1#2#3{{\setbox0=\hbox{$#1{#2#3}{\int}$} 
  \vcenter{\hbox{$#2#3$}}\kern-.5\wd0}} 
\def\-int{\Xint -}

\newcommand{\R}{\mathbb{R}}
\newcommand{\N}{\mathbb{N}}
\newcommand{\h}{H^{s}(\R^{N})}

\DeclareMathOperator{\dive}{div}

\DeclareMathOperator{\F}{\mathcal{F}}
\DeclareMathOperator{\erre}{\mathcal{R}}
\DeclareMathOperator{\C}{\mathcal{C}}

\newtheorem{prop}{Proposition}[section]
\newtheorem{lem}{Lemma}[section]
\newtheorem{thm}{Theorem}[section]

\begin{document}

\title[multiplicity result]{A multiplicity result for a fractional Kirchhoff equation in $\R^{N}$ with a general nonlinearity}

\author[V. Ambrosio]{Vincenzo Ambrosio}
\address{Vincenzo Ambrosio\hfill\break\indent 
Dipartimento di Scienze Pure e Applicate (DiSPeA),\hfill\break\indent
Universit\`a degli Studi di Urbino `Carlo Bo'\hfill\break\indent
Piazza della Repubblica, 13\hfill\break\indent
61029 Urbino (Pesaro e Urbino, Italy)}
\email{vincenzo.ambrosio@uniurb.it}

\author[T. Isernia]{Teresa Isernia}
\address{Teresa Isernia\hfill\break\indent
Dipartimento di Matematica e Applicazioni `Renato Caccioppoli' \hfill\break\indent
Universit\`a degli Studi di Napoli `Federico II' \hfill\break\indent
Via Cintia 1, 80126 Napoli\hfill\break\indent
Italy}
\email{teresa.isernia@unina.it}


\begin{abstract}
In this paper we deal with the following fractional Kirchhoff equation
\begin{equation*}
\left(p+q(1-s) \iint_{\R^{2N}} \frac{|u(x)- u(y)|^{2}}{|x-y|^{N+2s}} \, dx\,dy \right)(-\Delta)^{s}u = g(u) \mbox{ in } \R^{N},  
\end{equation*}
where $s\in (0,1)$, $N\geq 2$, $p>0$, $q$ is a small positive parameter and $g: \R\rightarrow \R$ is an odd function satisfying Berestycki-Lions type assumptions. 
By using minimax arguments, we establish a multiplicity result for the above equation, provided that $q$ is sufficiently small.  
\end{abstract}

\maketitle

\section{Introduction}

\noindent
In this paper we study the multiplicity of weak solutions to the following nonlinear fractional Kirchhoff equation 
\begin{equation}\label{K}
\left(p+q(1-s) \iint_{\R^{2N}} \frac{|u(x)- u(y)|^{2}}{|x-y|^{N+2s}} \, dx\, dy \right)(-\Delta)^{s}u = g(u)\, \mbox{ in } \R^{N} 
\end{equation}
where $s\in (0,1)$, $N\geq 2$, $p>0$, $q$ is a small positive parameter and $g$ is a nonlinearity which satisfies suitable assumptions. 
The operator $(-\Delta)^{s}$ is the fractional Laplacian which may be defined for a function $u$ belonging to the Schwartz space $ \mathcal{S}(\R^{N})$ of rapidly decaying functions as
$$
(-\Delta)^{s}u(x)=C_{N,s} \,P. V. \int_{\R^{N}} \frac{u(x)-u(y)}{|x-y|^{N+2s}} dy,  \quad x\in \R^{N}.
$$
The symbol P.V. stands for the Cauchy principal value and $C_{N,s}$ is a normalizing constant; see \cite{DPV} for more details.\\
When $s\rightarrow 1^{-}$ in \eqref{K}, from Theorem $2$ (and Corollary $2$) in \cite{BBM}, we can see that \eqref{K} becomes the following Kirchhoff equation
\begin{equation}\label{K1}
-\left(p+q \int_{\R^{N}} |\nabla u(x)|^{2} \, dx \right) \Delta u = g(u)\, \mbox{ in } \R^{N}, 
\end{equation}
which has been extensively studied in the last decade. \\
The equation (\ref{K1}) is related to the stationary analogue of the Kirchhoff equation
\begin{equation*}
u_{tt} - \left( p + q \int_{\Omega} |\nabla u(x)|^{2} \, dx \right) \Delta u=g(x,u)
\end{equation*}
with $\Omega \subset \R^{N}$ bounded domain, which was proposed by Kirchhoff in 1883 \cite{Kirchhoff} as an extension of the classical D'Alembert's wave equation 
\begin{equation}
\rho\, u_{tt} - \left( \frac{P_{0}}{h} + \frac{E}{2L} \int_{0}^{L} |u_{x}|^{2} \, dx \right) u_{xx}^{2}=g(x,u)
\end{equation}
for free vibrations of elastic strings. Kirchhoff's model takes into account the changes in length of the string produced by transverse vibrations. Here, $L$ is the length of the string, $h$ is the area of the cross section, $E$ is the Young modulus of the material, $\rho$ is the mass density and $P_{0}$ is the initial tension.
The early classical studies dedicated to Kirchhoff equations were given by Bernstein \cite{Bernstein} and Pohozaev \cite{Pohozaev}. However, equation (\ref{K1}) received much attention only after the paper by Lions \cite{Lions1}, 
where a functional analysis approach was proposed to attack it. For more recent results concerning Kirchhoff-type equations we refer to \cite{ACF, ACM, AuPu1, AUtPucSal, Fig, LS, MZ, PZ}.\\
On the other hand, a great attention has been recently focused on the study of nonlinear fractional Kirchhoff problem. 
In \cite{FV}, Fiscella and Valdinoci proposed an interesting interpretation of Kirchhoff's equation in the fractional setting, by proving the existence of nonnegative solutions for a critical Kirchhoff type problem in a bounded domain of $\R^{N}$. Subsequently, in \cite{AuFiPu} the authors investigated the existence and the asymptotic behavior of nonnegative solutions for a class of stationary Kirchhoff problems driven by a fractional integro-differential operator and involving a critical nonlinearity. Pucci and Saldi in \cite{PuSa} established the existence and multiplicity of nontrivial nonnegative entire solutions for a Kirchhoff type eigenvalue problem in $\R^{N}$ involving a critical nonlinearity and the fractional Laplacian. 
More recently, in \cite{FMS} has been proved the existence of infinitely many weak solutions for a Cauchy problem for a fractional Kirchhoff-type equation by using the genus theory of Krasnosel'skii; see also \cite{AuPu4, MMTZ, MR, MT, MV, PXZ} for related problems.

\smallskip

\noindent
Motivated by the above papers, in this work we aim to study the multiplicity of weak solutions to the fractional Kirchhoff equation (\ref{K}) with $q$ small parameter and $g$ is a general subcritical nonlinearity. 
More precisely, we suppose that $g:\R \rightarrow \R$ satisfies Berestycki-Lions type assumptions \cite{BL1, BL2}, that is:
\begin{compactenum}[($g_1$)]
\item $g\in \C^{1}(\R, \R)$ and odd;
\item $\displaystyle{-\infty < \liminf_{t\rightarrow 0^{+}} \frac{g(t)}{t} \leq \limsup_{t\rightarrow 0^{+}} \frac{g(t)}{t}=-m<0}$;
\item $\displaystyle{\lim_{t\rightarrow \pm \infty}\frac{|g(t)|}{|t|^{2^{*}_{s}-1}}= 0}$, where $\displaystyle{2^{*}_{s}= \frac{2N}{N-2s}}$;
\item there exists $\zeta>0$ such that $\displaystyle{G(\zeta):= \int_{0}^{\zeta} g(t) \, d t>0}$.
\end{compactenum}
Let us recall that when $q=0$ and $p=1$ in (\ref{K}), in \cite{A, CW} has been established the existence and multiplicity of radially symmetric solutions to the fractional scalar field problem 
\begin{equation}\label{FBL}
(-\Delta)^{s}u=g(u) \mbox{ in } \R^{N}.
\end{equation}
Now, we aim to study a generalization of (\ref{FBL}), and we look for weak solutions to (\ref{K}) with $q>0$ sufficiently small. \\ 
Our main result is the following:
\begin{thm}\label{thm1.2}
Let us suppose that ($g_1$), ($g_2$), ($g_3$) and ($g_4$) are satisfied. Then, for any $h\in \N$ there exists $q(h)>0$ such that for any $0<q<q(h)$ equation (\ref{K}) admits at least $h$ couples of solutions in $\h$ with radial symmetry. 
\end{thm}


A common approach to deal with nonlinear problems involving the fractional Laplacian, has been proposed by Caffarelli and Silvestre in \cite{CS}. It consists to realize $(-\Delta)^{s}$  as an operator that maps a Dirichlet boundary condition to a Neumann-type
condition via an extension problem on the upper half-space $\R^{N+1}_{+}$.
More precisely, for $u\in \h$ one considers the problem 
\begin{equation*}
\left\{
\begin{array}{ll}
-\dive (y^{1-2s} \nabla v)=0 &\mbox{ in } \R^{N+1}_{+}\\
v(x, 0)=u(x) &\mbox{ on } \partial \R^{N+1}_{+}
\end{array}
\right.
\end{equation*}
from where the fractional Laplacian is obtained as 
\begin{equation*}
(-\Delta)^{s} u(x)= -\kappa_{s} \lim_{y\rightarrow 0^{+}} y^{1-2s} v_{y}(x, y),
\end{equation*}
where $\kappa_{s}$ is a suitable constant and the equality holds in distributional sense. \\
In this paper we investigate the problem (\ref{K}) directly in $\h$ in order to adapt the techniques developed in the classical case $s=1$.\\
More precisely, we follow the ideas in \cite{ADP}, and by combining the Mountain Pass approach introduced in \cite{HIT} with the truncation argument of \cite{JL}, we prove the multiplicity result above stated. 
\smallskip

\noindent
The paper is organized as follows: in Sec. $2$ some notations and preliminaries are given, including lemmas that are required to obtain our main Theorem; in Sec. $3$ we establish an abstract critical point result and finally in Sec. $4$ we provide the proof of Theorem \ref{thm1.2}.

\section{Preliminaries}
\noindent
For any $s\in (0,1)$ we define the fractional Sobolev spaces
$$
H^{s}(\R^{N})= \left\{u\in L^{2}(\R^{N}) : \frac{|u(x)-u(y)|}{|x-y|^{\frac{N+2s}{2}}} \in L^{2}(\R^{2N}) \right \}
$$
endowed with the natural norm 
$$
\|u\|_{H^{s}(\R^{N})} = \sqrt{[u]_{H^{s}(\R^{N})}^{2} + \|u\|_{L^{2}(\R^{N})}^{2}}
$$
where the so-called Gagliardo seminorm of $u$ is given by
$$
[u]^{2}_{H^{s}(\R^{N})} =\iint_{\R^{2N}} \frac{|u(x)-u(y)|^{2}}{|x-y|^{N+2s}} \, dx \, dy. 
$$

\noindent
For the reader's convenience, we review the main embedding result for this class of fractional Sobolev spaces. 
\begin{thm}\cite{DPV}\label{sobolevthm}
Let $s\in (0,1)$ and $N>2s$. Then $H^{s}(\R^{N})$ is continuously embedded in $L^{q}(\R^{N})$ for any $q\in [2, 2^{*}_{s}]$ and compactly in $L^{q}_{\rm loc}(\R^{N})$ for any $q\in [2, 2^{*}_{s})$. 
\end{thm}

\noindent
Let us introduce 
$$
H^{s}_{\rm rad}(\R^{N})=\left\{u\in H^{s}(\R^{N}): u(x)=u(|x|)\right\}
$$
the space of radial functions in $H^{s}(\R^{N})$. For this space it holds the following compactness result due to Lions \cite{Lions}:
\begin{thm}\label{Lions}\cite{Lions}
Let $s\in (0,1)$ and $N\geq 2$. Then $H^{s}_{\rm rad}(\R^{N})$ is compactly embedded in $L^{q}(\R^{N})$ for any $q\in (2, 2^{*}_{s})$.
\end{thm}

\noindent
Finally, we recall the following fundamental compactness results:
\begin{lem}\cite{ADP, BL1}\label{prop2.5}
Let $P$ and $Q:\R\rightarrow\R$ be a continuous functions satisfying
\begin{equation*}
\lim_{t\rightarrow +\infty} \frac{P(t)}{Q(t)}=0,
\end{equation*}
$\{v_{j}\}_{j\in \N}$, $v$ and $w$ be measurable functions from $\R^{N}$ to $\R$, with $w$ bounded, such that 
\begin{align*}
&\sup_{j\in \N} \int_{\R^{N}} |Q(v_{j}(x)) w| \, dx <+\infty, \\
&P(v_{j}(x))\rightarrow v(x) \mbox{ a.e. in } \R^{N}. 
\end{align*}
Then $\|(P(v_{j})-v)w\|_{L^{1}(\mathcal{B})}\rightarrow 0$, for any bounded Borel set $\mathcal{B}$. \\
Moreover, if we have also
\begin{equation*}
\lim_{t\rightarrow 0} \frac{P(t)}{Q(t)}=0,
\end{equation*}
and
\begin{equation*}
\lim_{|x| \rightarrow \infty} \sup_{j\in \N} |v_{j}(x)|=0,
\end{equation*}
then $\|(P(v_{j})-v)w\|_{L^{1}(\R^{N})}\rightarrow 0$.
\end{lem}

\begin{lem}\cite{CW}\label{CW}
Let  $(X, \| \cdot \|)$ be a Banach space such that $X$ is embedded respectively continuously and compactly into $L^{q}(\R^{N})$ for $q\in [q_{1}, q_{2}]$ and $q\in (q_{1}, q_{2})$, where $q_{1}, q_{2}\in (0, \infty)$.
Assume that $\{v_{j}\}_{j\in \N}\subset X$, $v: \R^{N} \rightarrow \R$ is a measurable function and $P\in \C(\R, \R)$ is such that
\begin{compactenum}[(i)]
\item $\displaystyle{\lim_{|t|\rightarrow 0} \frac{P(t)}{|t|^{q_{1}}}=0}$, \\
\item $\displaystyle{\lim_{|t|\rightarrow \infty} \frac{P(t)}{|t|^{q_{2}}}=0}$,\\
\item $\displaystyle{\sup_{j\in \N} \|v_{j}\|<\infty}$,\\
\item $\displaystyle{\lim_{j \rightarrow \infty} P(v_{j}(x))=v(x)} \mbox{ for a.e. } x\in \R^{N}$.
\end{compactenum}
Then, up to a subsequence, we have
$$
\lim_{j\rightarrow \infty} \|P(v_{j})-v\|_{L^{1}(\R^{N})}=0.
$$
\end{lem}

\section{A critical point result}
\noindent
In this section we provide an abstract multiplicity result which allows us to prove Theorem \ref{thm1.2}.
Let us introduce the following functional defined for $u\in H^{s}(\R^{N})$
\begin{equation}\label{2}
\F_{q}(u) = \frac{1}{2} [u]_{\h}^{2} + q \erre(u) - \int_{\R^{N}} G(u) \, dx, 
\end{equation}
where $q>0$ is a small parameter and $\erre:H^{s}(\R^{N}) \rightarrow \R$. \\
We suppose that 
$$
\erre=\sum_{i=1}^{k} \erre_{i}
$$ 
and, for each $i=1, \dots, k$ the functional $\erre_{i}$ satisfies
\begin{compactenum}[($\erre_1$)]
\item $\erre_{i}\in \C^{1}(H^{s}(\R^{N}), \R)$ is nonnegative and even;
\item there exists $\delta_{i}>0$ such that $\langle \erre_{i}'(u) , u \rangle\leq C \|u\|_{\h}^{\delta_{i}}$ for any $u\in H^{s}(\R^{N})$;
\item if $\{u_{j}\}_{j\in \N}\subset H^{s}(\R^{N})$ is weakly convergent to $u\in H^{s}(\R^{N})$, then 
$$
\limsup_{j\rightarrow \infty} \, \langle \erre_{i}'(u_{j}) , u-u_{j} \rangle \leq 0;
$$
\item there exist $\alpha_{i}, \beta_{i}\geq 0$ such that if $u\in H^{s}(\R^{N})$, $t>0$ and $u_{t}=u\left(\frac{\cdot}{t}\right)$, then
$$
\erre_{i}(u_{t}) = t^{\alpha_{i}} \erre_{i}(t^{\beta_{i}}u);
$$ 
\item $\erre_{i}$ is invariant under the action of the $N$-dimensional orthogonal group, i.e. $\erre_{i}(u(g\cdot))= \erre_{i}(u(\cdot))$ for every $g\in \mathcal{O}(N)$. 
\end{compactenum}
Let us observe that for any $u\in H^{s}(\R^{N})$, $\displaystyle{\erre_{i}(u)-\erre_{i}(0)= \int_{0}^{1} \frac{d}{dt}\erre_{i}(tu)\, dt}$, so by the assumption ($\erre_2$) we have
\begin{equation}\label{10}
\erre_{i}(u)\leq C_{1}+C_{2}\|u\|_{\h}^{\delta_{i}}. 
\end{equation}

The main result of this section is the following
\begin{thm}\label{thm1.1}
Let us suppose $(g_1)-(g_4)$ and $(\erre_1)-(\erre_5)$. Then, for any $h\in \N$ there exists $q(h)>0$ such that for any $0<q<q(h)$ the functional $\F_{q}$ admits at least $h$ couples of critical points in $H^{s}_{\rm rad}(\R^{N})$. 
\end{thm}
\smallskip
\noindent 
Let us define, for any $t\geq 0$, 
\begin{align*}
&g_{1}(t):=(g(t)+mt)^{+}\\
&g_{2}(t):=g_{1}(t)-g(t), 
\end{align*}
and we extend them as odd functions for $t\leq 0$. Observing that 
\begin{align}
&\lim_{t\rightarrow 0} \frac{g_{1}(t)}{t}=0, \label{4}\\
&\lim_{t\rightarrow \infty} \frac{g_{1}(t)}{t^{2^{*}_{s}-1}}=0, \label{5}\\
&g_{2}(t)\geq mt \quad \forall t\geq 0,  \label{g2}
\end{align}
we deduce that for any $\varepsilon>0$ there exists $C_{\varepsilon}>0$ such that 
\begin{equation}\label{g1}
g_{1}(t)\leq C_{\varepsilon} t^{2^{*}_{s}-1} + \varepsilon g_{2}(t) \quad \forall t\geq 0. 
\end{equation}
Setting
\begin{equation*}
G_{i}(t):=\int_{0}^{t} g_{i}(\tau) \, d\tau \quad i=1,2, 
\end{equation*}
by (\ref{g2}) immediately follows that
\begin{equation}\label{8}
G_{2}(t)\geq \frac{m}{2}t^{2} \quad \forall t\in\R,
\end{equation} 
and, by (\ref{g1}) we can see that for any $\varepsilon>0$ there exists $C_{\varepsilon}>0$ such that
\begin{equation}\label{9}
G_{1}(t)\leq C_{\varepsilon}\, |t|^{2^{*}_{s}} + \varepsilon \,G_{2}(t) \quad \forall t\in \R.
\end{equation} 

\noindent
In view of ($\erre_5$), all functionals that we will consider along the paper are invariant under rotations, so, from now on, we will directly define our functionals in $H^{s}_{\rm rad}(\R^{N})$. 

\noindent
Following \cite{JL}, let $\chi \in \C^{\infty}([0, +\infty), \R)$ be a cut-off function such that
\begin{equation*} 
\left\{
\begin{array}{ll}
\chi(t)=1 &\mbox{ for } t\in [0,1]\\
0\leq \chi(t)\leq 1 &\mbox{ for } t\in (1, 2) \\
\chi(t)=0 &\mbox{ for } t\in [2, +\infty) \\
\|\chi'\|_{\infty}\leq 2,
\end{array}
\right.
\end{equation*}
and we set
$$
\xi_{\Lambda}(u)= \chi \left(\frac{\|u\|_{\h}^{2}}{\Lambda^{2}} \right). 
$$
Then we introduce the truncated functional $\F_{q}^{\Lambda}: H^{s}_{\rm rad}(\R^{N})\rightarrow \R$ defined as
\begin{equation*}
\F_{q}^{\Lambda}(u)=\frac{1}{2}[u]_{\h}^{2}+ q \, \xi_{\Lambda}(u) \erre(u) -\int_{\R^{N}}G(u)\, dx . 
\end{equation*}
Clearly, a critical point $u$ of $\F_{q}^{\Lambda}$ with $\|u\|_{\h}\leq \Lambda$ is a critical point of $\F_{q}$. \\
Our first aim is to prove that the truncated functional $\F_{q}^{\Lambda}$ has a symmetric mountain pass geometry:

\begin{lem}\label{lem2.1}
There exist $r_{0}>0$ and $\rho_{0}>0$ such that 
\begin{align}\begin{split}\label{11}
&\F_{q}^{\Lambda}(u)\geq 0, \quad \mbox{ for } \|u\|_{\h}\leq r_{0}  \\ 
&\F_{q}^{\Lambda}(u)\geq \rho_{0}, \quad \mbox{ for } \|u\|_{\h}=r_{0}. 
\end{split}\end{align}
Moreover, for any $n\in\N$ there exists an odd continuous map 
$$
\gamma_{n}: \mathbb{S}^{n-1}\rightarrow H^{s}_{\rm rad}(\R^{N})
$$ 
such that 
\begin{equation}\label{f}
\F_{q}^{\Lambda}(\gamma_{n}(\sigma))<0 \quad \mbox{ for all } \sigma \in \mathbb{S}^{n-1},
\end{equation}
where 
$$
\mathbb{S}^{n-1}=\{\sigma= (\sigma_1, \dots, \sigma_n) \in \R^{n} : |\sigma|=1\}. 
$$
\end{lem} 
\begin{proof}
Taking $\varepsilon=\frac{1}{2}$ in (\ref{9}), and by using (\ref{8}), the positivity of $\erre$, and Theorem \ref{sobolevthm}, we have
\begin{align*}
\F_{q}^{\Lambda}(u) &= \frac{1}{2}[u]^{2}_{\h} + \int_{\R^{N}} G_{2}(u)\, dx + q \,\xi_{\Lambda}(u) \erre(u) -\int_{\R^{N}} G_{1}(u) \, dx \\
&\geq \frac{1}{2}[u]^{2}_{\h} + \frac{m}{4}\|u\|_{L^{2}(\R^{N})}^{2} - C_{\frac{1}{2}} \|u\|_{L^{2^{*}_{s}}(\R^{N})}^{2^{*}_{s}} \\
&\geq \min\left\{\frac{1}{2}, \frac{m}{4}\right\} \|u\|_{\h}^{2} - C_{\frac{1}{2}} C^{*}\|u\|_{\h}^{2^{*}_{s}}
\end{align*}
from which easily follows $(\ref{11})$. \\
Proceeding similarly to Theorem 10 in \cite{BL2}, for any $n\in \N$, there exists an odd continuous map $\pi_{n}: \mathbb{S}^{n-1}\rightarrow H^{s}_{\rm rad}(\R^{N})$ such that
\begin{align*}
&0\notin \pi_{n}(\mathbb{S}^{n-1}),\\
&\int_{\R^{N}} G(\pi_{n}(\sigma))\, dx \geq 1 \mbox{ for all } \sigma \in \mathbb{S}^{n-1}. 
\end{align*}
Let us define
$$
\psi_{n}^{t}(\sigma)= \pi_{n}(\sigma) \left(\frac{\cdot}{t}\right) \mbox{ with } t\geq 1. 
$$
Then, for $t$ sufficiently large, we get 
\begin{align*}
\F_{q}^{\Lambda}(\psi_{n}^{t}(\sigma))&= \frac{t^{N-2s}}{2}[\pi_{n}(\sigma)]_{\h}^{2} \\
&+ q \, \chi\left(\frac{t^{N-2s}[\pi_{n}(\sigma)]_{\h}^{2} +t^{N}\|\pi_{n}(\sigma)\|_{L^{2}(\R^{N})}^{2} }{\Lambda^{2}} \right) \erre(\psi_{n}^{t}(\sigma)) \\
&- t^{N} \int_{\R^{N}} G(\pi_{n}(\sigma)) \, dx \\
&\leq t^{N-2s} \left\{\frac{[\pi_{n}(\sigma)]_{\h}^{2}}{2} - t^{2s}\right\}<0. 
\end{align*}
Therefore, we can choose $\bar{t}$ such that $\F_{q}^{\Lambda}(\psi_{n}^{\bar t}(\sigma))<0$ for all $\sigma \in \mathbb{S}^{n-1}$, and by setting $\gamma_{n}(\sigma)(x):=\psi_{n}^{\bar{t}}(\sigma)(x)$, we can see that $\gamma_{n}$ satisfies the required properties. 

\end{proof}

\noindent
Now we define the minimax value of $\F_{q}^{\Lambda}$ by using the maps $\gamma_{n}: \partial\mathcal{D}_{n}\rightarrow H^{s}_{\rm rad}(\R^{N})$ obtained in Lemma \ref{lem2.1}. For any $n\in \N$, let
$$
b_{n}= b_{n}(q, \Lambda) = \inf_{\gamma\in \Gamma_{n}} \max_{\sigma \in \mathcal{D}_{n}} \F_{q}^{\Lambda}(\gamma(\sigma)),  
$$
where $\mathcal{D}_{n}=\{\sigma= (\sigma_1, \dots, \sigma_n)\in \R^{n} : |\sigma|\leq 1\}$ and 
\begin{equation*}
\Gamma_{n}= \left\{ \gamma\in \C(\mathcal{D}_{n}, H^{s}_{\rm rad}(\R^{N})) : 
\begin{array}{ll}
\gamma(-\sigma)=-\gamma(\sigma) &\mbox{ for all } \sigma \in \mathcal{D}_{n} \\
\gamma(\sigma)=\gamma_{n}(\sigma) &\mbox{ for all } \sigma \in \partial \mathcal{D}_{n}	
\end{array}
\right\}.
\end{equation*}
Let us introduce the following modified functionals
\begin{align*}
&\widetilde{\F_{q}}(\theta, u)= \F_{q}(u(\cdot/ e^{\theta}))\\
&\widetilde{\F_{q}^{\Lambda}}(\theta, u)= \F_{q}^{\Lambda}(u(\cdot/e^{\theta}))
\end{align*}
for $(\theta,u)\in \R\times H^{s}_{\rm rad}(\R^{N})$.\\
We set
\begin{align*}
&\widetilde{\F_{q}'}(\theta, u)=\frac{\partial}{\partial u} \widetilde{\F_{q}}(\theta, u), \\
&(\widetilde{\F_{q}^{\Lambda}})'(\theta, u)= \frac{\partial}{\partial u} \widetilde{\F_{q}^{\Lambda}}(\theta, u), \\
&\tilde{b}_{n}= \tilde{b}_{n}(q, \Lambda)= \inf_{\tilde{\gamma}\in \widetilde{\Gamma}_{n}} \max_{\sigma \in \mathcal{D}_{n}} \widetilde{\F_{q}^{\Lambda}}(\tilde{\gamma}(\sigma)),
\end{align*}
where 
\begin{equation*}
\widetilde{\Gamma}_{n}= \left\{ \tilde{\gamma}\in \C(\mathcal{D}_{n}, \R\times H^{s}_{\rm rad}(\R^{N})) : 
\begin{array}{ll}
\tilde{\gamma}(\sigma)= (\theta(\sigma), \eta(\sigma)) \mbox{ satisfies }\\
(\theta(-\sigma), \eta(-\sigma))=(\theta(\sigma), -\eta(\sigma)) &\mbox{ for all } \sigma \in \mathcal{D}_{n} \\
(\theta(\sigma), \eta(\sigma))= (0, \gamma_{n}(\sigma)) &\mbox{ for all } \sigma \in \partial \mathcal{D}_{n}	
\end{array}
\right\}.
\end{equation*}
By the assumption ($\erre_4$) we get
\begin{align*}
\widetilde{\F_{q}}(\theta, u)=\frac{e^{(N-2s)\theta}}{2} [u]_{\h}^{2} + q\sum_{i=1}^{k} e^{\alpha_{i}\theta}\erre_{i} (e^{\beta_{i}\theta}u)- e^{N\theta} \int_{\R^{N}} G(u)\, dx , 
\end{align*}
and 
\begin{align*}
\widetilde{\F_{q}^{\Lambda}}(\theta, u)&= \frac{e^{(N-2s)\theta}}{2} [u]_{\h}^{2} + q \chi \left(\frac{e^{(N-2s)\theta}[u]_{\h}^{2}+e^{N\theta}\|u\|_{L^{2}(\R^{N})}^{2}}{\Lambda^{2}}\right) \sum_{i=1}^{k} e^{\alpha_{i}\theta}\erre_{i} (e^{\beta_{i}\theta}u)\\
&- e^{N\theta} \int_{\R^{N}} G(u) \, dx. 	
\end{align*}

\noindent
Proceeding as in \cite{A, HIT, Rab}, we can see that the following results hold.
\begin{lem}\label{lem2.2}
We have
\begin{compactenum}[(1)]
\item there exists $\bar{b}>0$ such that $b_{n}\geq \bar{b}$ for any $n\in \N$, 
\item $b_{n}\rightarrow +\infty$,
\item $b_{n}=\tilde{b}_{n}$ for any $n\in \N$. 
\end{compactenum}
\end{lem}

\begin{lem}\label{lem2.3}
For any $n\in \N$ there exists a sequence $\{(\theta_{j}, u_{j})\}_{j\in \N}\subset \R \times H^{s}_{\rm rad}(\R^{N})$ such that
\begin{compactenum}[(i)]
\item $\theta_{j}\rightarrow 0$,
\item $\widetilde{\F_{q}^{\Lambda}}(\theta_{j}, u_{j})\rightarrow b_{n}$,
\item $(\widetilde{\F_{q}^{\Lambda}})'(\theta_{j}, u_{j})\rightarrow 0$ strongly in $(H^{s}_{\rm rad}(\R^{N}))^{-1}$,
\item $\displaystyle{\frac{\partial}{\partial \theta} \widetilde{\F_{q}^{\Lambda}}(\theta_{j}, u_{j})\rightarrow 0}$.
\end{compactenum}
\end{lem}

\noindent
Our goal is to prove that, for a suitable choice of $\Lambda$ and $q$, the sequence $\{(\theta_{j}, u_{j})\}_{j\in \N}$ given by Lemma \ref{lem2.3} is a bounded Palais-Smale sequence for $\F_{q}$. We begin proving the boundedness of $\{u_{j}\}_{j\in\N}$ in $\h$.  
\begin{prop}\label{prop2.4}
Let $n\in \N$ and $\Lambda_{n}>0$ sufficiently large. There exists $q_{n}$, depending on $\Lambda_{n}$, such that for any $0<q<q_{n}$, if $\{(\theta_{j}, u_{j})\}_{j\in \N}\subset \R\times H^{s}_{\rm rad}(\R^{N})$ is the sequence given in Lemma \ref{lem2.3}, then, up to a subsequence, $\|u_{j}\|_{\h}\leq \Lambda_{n}$, for any $j\in \N$. 
\end{prop}
\begin{proof}
Taking into account Lemma \ref{lem2.2} and Lemma \ref{lem2.3} we have
\begin{equation*}
N \widetilde{\F_{q}^{\Lambda}}(\theta_{j}, u_{j}) - \frac{\partial}{\partial \theta} \widetilde{\F_{q}^{\Lambda}}(\theta_{j}, u_{j}) = Nb_{n} +o_{j}(1),
\end{equation*}
which can be written as 
\begin{align}\label{13}
&se^{(N-2s)\theta_{j}} [u_{j}]^{2}_{\h} \nonumber \\
&= q \,\chi \left( \frac{\|u_{j}(\cdot/e^{\theta_{j}})\|_{\h}^{2}}{\Lambda^{2}} \right) \sum_{i=1}^{k}(\alpha_{i}-N)\erre_{i}(u_{j}(\cdot/e^{\theta_{j}}))\nonumber \\
&+ q \,\chi \left( \frac{\|u_{j}(\cdot/e^{\theta_{j}})\|_{\h}^{2}}{\Lambda^{2}} \right) \sum_{i=1}^{k} e^{\alpha_{i}\theta_{j}}\langle \erre_{i}'(e^{\beta_{i}\theta_{j}}u_{j}), \beta_{i}e^{\beta_{i}\theta_{j}}u_{j}\rangle \nonumber \\
&+ q \,\chi' \left( \frac{\|u_{j}(\cdot/e^{\theta_{j}})\|_{\h}^{2}}{\Lambda^{2}} \right) \frac{(N-2s)e^{(N-2s)\theta_{j}}[u_{j}]_{\h}^{2} + Ne^{N\theta_{j}} \|u_{j}\|_{L^{2}(\R^{N})}^{2}}{\Lambda^{2}} \erre(u_{j}(\cdot/e^{\theta_{j}}))\nonumber  \\
&+Nb_{n} +o_{j}(1) \nonumber \\
&=:\mathcal{I}_{j}+\mathcal{II}_{j}+\mathcal{III}_{j}+Nb_{n}+ o_{j}(1).
\end{align}
By the definition of $b_{n}$, if $\gamma \in \Gamma_{n}$, we deduce that
\begin{align}
b_{n}&\leq \max_{\sigma \in \mathcal{D}_{n}} \F_{q}^{\Lambda}(\gamma(\sigma)) \nonumber \\
&\leq \max_{\sigma \in \mathcal{D}_{n}} \left\{\frac{1}{2}[\gamma(\sigma)]^{2}_{\h} -\int_{\R^{N}}G(\gamma(\sigma)) \, dx\right\} + \max_{\sigma \in \mathcal{D}_{n}} \left(q \, \xi_{\Lambda}(\gamma(\sigma)) \erre(\gamma(\sigma)) \right) \nonumber \\
&=: A_{1} + A_{2}(\Lambda). \label{bn}
\end{align}
Now, if $\|\gamma(\sigma)\|_{\h}^{2}\geq 2\Lambda^{2}$ then $A_{2}(\Lambda)=0$, otherwise, by (\ref{10}), we can find $\delta>0$ such that
\begin{equation*}
A_{2}(\Lambda)\leq q \left(C_{1} + C_{2} \|\gamma(\sigma)\|_{\h}^{\delta} \right) \leq q \left( C_{1} + C_{2}' \Lambda^{\delta}\right). 
\end{equation*}
In addition we have the following estimates:
\begin{align}
&|\mathcal{I}_{j}| \leq q \left( C_{3} + C_{4} \Lambda^{\delta}\right), \label{nn}\\
&|\mathcal{II}_{j}| \leq C_{5} \, q\, \Lambda^{\delta}, \label{14}\\
&|\mathcal{III}_{j}|\leq q \left( C_{6} + C_{7} \Lambda^{\delta}\right). \label{15}
\end{align}
Putting together (\ref{bn}), (\ref{nn}), (\ref{14}) and (\ref{15}), from (\ref{13}) we obtain
\begin{equation}\label{16}
[u_{j}]_{\h}^{2} \leq C' + q \left( C_{8} + C_{9} \Lambda^{\delta}\right). 
\end{equation}
On the other hand, by $(iv)$ of Lemma \ref{lem2.3} and (\ref{9}), we deduce that 
\begin{align}\label{17}
&\frac{(N-2s)e^{(N-2s)\theta_{j}}}{2} [u_{j}]_{\h}^{2} + q \,\chi \left( \frac{\|u_{j}(\cdot/e^{\theta_{j}})\|_{\h}^{2}}{\Lambda^{2}} \right) \sum_{i=1}^{k} \alpha_{i} \erre_{i}(u_{j} (\cdot/e^{\theta_{j}})) \nonumber \\
&+ q \,\chi \left( \frac{\|u_{j}(\cdot/e^{\theta_{j}})\|_{\h}^{2}}{\Lambda^{2}} \right) \sum_{i=1}^{k} e^{\alpha_{i}\theta_{j}}\langle \erre_{i}'(e^{\beta_{i}\theta_{j}}u_{j}), \beta_{i}e^{\beta_{i}\theta_{j}}u_{j}\rangle \nonumber \\
&+ q \,\chi' \left( \frac{\|u_{j}(\cdot/e^{\theta_{j}})\|_{\h}^{2}}{\Lambda^{2}} \right) \frac{(N-2s)e^{(N-2s)\theta_{j}}[u_{j}]_{\h}^{2} + Ne^{N\theta_{j}} \|u_{j}\|_{L^{2}(\R^{N})}^{2}}{\Lambda^{2}} \erre(u_{j}(\cdot/e^{\theta_{j}})) \nonumber \\
&+Ne^{N\theta_{j}} \int_{\R^{N}} G_{2}(u_{j}) \, dx \nonumber \\
&=N e^{N\theta_{j}} \int_{\R^{N}} G_{1}(u_{j}) \, dx + o_{j}(1)\nonumber \\
&\leq Ne^{N\theta_{j}} \left( C_{\varepsilon} \int_{\R^{N}} |u_{j}|^{2^{*}_{s}} \, dx+ \varepsilon \int_{\R^{N}} G_{2}(u_{j}) \, dx\right) +o_{j}(1). 
\end{align}
Then, by using (\ref{8}), (\ref{14}), (\ref{15}), (\ref{16}), (\ref{17}) and Theorem \ref{sobolevthm}, we can infer
\begin{align}\label{18}
&\frac{Ne^{N\theta_{j}} m(1-\varepsilon)}{2} \int_{\R^{N}} u_{j}^{2} \, dx \nonumber \\
&\leq (1-\varepsilon) N e^{N\theta_{j}}\int_{\R^{N}} G_{2}(u_{j})\, dx \nonumber \\
&\leq N e^{N\theta_{j}} C_{\varepsilon} \int_{\R^{N}} |u_{j}|^{2^{*}_{s}} \, dx - q\, \chi \left( \frac{\|u_{j}(\cdot/e^{\theta_{j}})\|_{\h}^{2}}{\Lambda^{2}} \right) \sum_{i=1}^{k} e^{\alpha_{i}\theta_{j}}\langle \erre_{i}'(e^{\beta_{i}\theta_{j}}u_{j}), \beta_{i}e^{\beta_{i}\theta_{j}}u_{j}\rangle \nonumber \\
&- q \,\chi' \left( \frac{\|u_{j}(\cdot/e^{\theta_{j}})\|_{\h}^{2}}{\Lambda^{2}} \right) \frac{(N-2s)e^{(N-2s)\theta_{j}}[u_{j}]_{\h}^{2} + Ne^{N\theta_{j}} \|u_{j}\|_{L^{2}(\R^{N})}^{2}}{\Lambda^{2}} \erre(u_{j}(\cdot/e^{\theta_{j}}))+ o_{j}(1) \nonumber \\
&\leq C_{10} \left( [u_{j}]_{\h}^{2} \right)^{\frac{2^{*}_{s}}{2}} + q \left( C_{11} + C_{12}\Lambda^{\delta} \right) + o_{j}(1) \nonumber \\
&\leq C_{10} \left( C' + q \left( C_{8} + C_{9}\Lambda^{\delta} \right) \right)^{\frac{2^{*}_{s}}{2}} + q \left( C_{11} + C_{12}\Lambda^{\delta} \right) + o_{j}(1). 
\end{align}
Now, we argue by contradiction. If we suppose that there exists no subsequence $\{u_{j}\}_{j\in \N}$ which is uniformly bounded by $\Lambda$ in the $H^{s}$-norm, we can find $j_{0}\in \N$ such that 
\begin{equation}\label{19}
\|u_{j}\|_{\h}>\Lambda \mbox{ for all } j\geq j_{0}. 
\end{equation}
Without any loss of generality, we can assume that (\ref{19}) is true for any $u_{j}$. 
As a consequence, by using (\ref{16}), \eqref{18} and (\ref{19}), we can deduce that
\begin{equation*}
\Lambda^{2} < \|u_{j}\|_{\h}^{2} \leq C_{13} + C_{14} q\, \Lambda^{\frac{2^{*}_{s}}{2}\delta}
\end{equation*}
which is impossible for $\Lambda$ large and $q$ small enough. Indeed, to see this, we can observe that it is possible to find $\Lambda_{0}$ such that $\Lambda_{0}^{2}>C_{13}+1$ and $q_{0}=q_{0}(\Lambda_{0})$ such that $C_{14} q\, \Lambda^{\frac{2^{*}_{s}}{2}\delta}<1$, for any $q<q_{0}$, and this gives a contradiction.

\end{proof}

\noindent
At this point, we prove the following compactness result:

\begin{lem}\label{lem2.6}
Let $n\in \N$, $\Lambda_{n},q_{n}>0$ as in Proposition \ref{prop2.4} and $\{(\theta_{j}, u_{j})\}_{j\in \N} \subset \R \times H^{s}_{\rm rad}(\R^{N})$ be the sequence given in Lemma \ref{lem2.3}. Then $\{u_{j}\}_{j\in \N}$ admits a subsequence which converges in $H^{s}_{\rm rad}(\R^{N})$ to a nontrivial critical point of $\F_{q}$ at the level $b_{n}$. 
\end{lem}
\begin{proof}
By Proposition \ref{prop2.4}, we know that $\{u_{j}\}_{j\in \N}$ is bounded, so, by using Theorem \ref{Lions}, we can suppose, up to a subsequence, that there exists $u\in H^{s}_{\rm rad}(\R^{N})$ such that
\begin{align}\begin{split}\label{20}
&u_{j}\rightharpoonup u \, \mbox{ weakly in } H^{s}_{\rm rad}(\R^{N}),\\
&u_{j}\rightarrow u \, \mbox{ in } L^{p}(\R^{N}), \, 2<p<2^{*}_{s},\\
&u_{j}\rightarrow u \, \mbox{ a.e. in } \R^{N}. 
\end{split}\end{align}
By the weak lower semicontinuity we know that
\begin{equation}\label{21}
[u]_{\h}^{2} \leq \liminf_{j\rightarrow \infty} [u_{j}]_{\h}^{2}. 
\end{equation}
Recalling that $\|u_{j}\|_{\h}\leq \Lambda_{n}$ for any $j\in \N$, we can see that, for every $v\in H^{s}_{\rm rad}(\R^{N})$, 
\begin{align}\label{vince}
\langle \widetilde{\F'_{q}}(\theta_{j}, u_{j}), v\rangle &= \langle(\widetilde{\F_{q}^{\Lambda_{n}}})'(\theta_{j}, u_{j}),v\rangle \nonumber \\
&=e^{(N-2s)\theta_{j}} \iint_{\R^{2N}} \frac{u_{j}(x)-u_{j}(y) }{|x-y|^{N+2s}} (v(x)-v(y)) \, dxdy \nonumber \\
&+ q \sum_{i=1}^{k} e^{(\alpha_{i}+\beta_{i})\theta_{j}} \langle \erre_{i}' (e^{\beta_{i}\theta_{j}}u_{j}), v \rangle \nonumber \\
&+e^{N\theta_{j}} \int_{\R^{N}} g_{2}(u_{j})v \, dx - e^{N\theta_{j}}\int_{\R^{N}}g_{1}(u_{j})v \, dx.
\end{align} 
Taking into account (\ref{vince}) and $(iii)$ of Lemma \ref{lem2.3} we have
\begin{align}\label{22}
o_{j}(1)&=\langle \widetilde{\F_{q}'}(\theta_{j}, u_{j}), u\rangle - \langle \widetilde{\F_{q}'}(\theta_{j}, u_{j}), u_{j}\rangle \nonumber  \\
&=e^{(N-2s)\theta_{j}} \iint_{\R^{2N}} \frac{u_{j}(x) -u_{j}(y)}{|x-y|^{N+2s}} [(u(x)-u(y))-(u_{j}(x)-u_{j}(y))] \, dxdy \nonumber \\
&+ q \sum_{i=1}^{k} e^{(\alpha_{i}+\beta_{i})\theta_{j}} \langle \erre_{i}' (e^{\beta_{i}\theta_{j}}u_{j}), u-u_{j}\rangle \nonumber \\
&+e^{N\theta_{j}} \int_{\R^{N}} g_{2}(u_{j})(u-u_{j})\, dx - e^{N\theta_{j}}\int_{\R^{N}}g_{1}(u_{j})(u-u_{j}) \, dx.
\end{align}
Now, by applying the first part of Lemma \ref{prop2.5} for $P(t)=g_{i}(t)$, $i=1,2$, $Q(t)=|t|^{2^{*}_{s} -1}$, $v_{j}= u_{j}$, $v=g_{i}(u)$, $i=1,2$ and $w \in \C^{\infty}_{0}(\R^{N})$, by $(g_3)$, (\ref{5}) and (\ref{20}) we can see, as $j\rightarrow \infty$
\begin{equation*}
\int_{\R^{N}} g_{i}(u_{j})w \, dx \rightarrow \int_{\R^{N}} g_{i}(u) w  \, dx \quad i=1,2, 
\end{equation*}
so we obtain
\begin{equation}\label{23}
\int_{\R^{N}} g_{i}(u_{j})u  \, dx \rightarrow \int_{\R^{N}} g_{i}(u) u  \, dx \quad i=1,2. 
\end{equation}
Taking $X=\h$, $q_{1}=2$, $q_{2}=2^{*}_{s}$, $v_{j}=u_{j}$, $v=g_{1}(u)u$ and $P(t)=g_{1}(t)t$ in Lemma \ref{CW}, by (\ref{4}), (\ref{5}) and (\ref{20}) we deduce
\begin{equation}\label{24}
\int_{\R^{N}} g_{1}(u_{j})u_{j}  \, dx \rightarrow \int_{\R^{N}} g_{1}(u) u  \, dx. 
\end{equation}
On the other hand, (\ref{20}) and Fatou's Lemma yield
\begin{equation}\label{25}
\int_{\R^{N}} g_{2}(u)u  \, dx \leq \liminf_{j\rightarrow \infty} \int_{\R^{N}} g_{2}(u_{j})u_{j}  \, dx.
\end{equation}
Putting together (\ref{22}), (\ref{23}), (\ref{24}), (\ref{25}), and by using ($\erre_3$) we get
\begin{align}\label{26}
\limsup_{j\rightarrow \infty} \, [u_{j}]_{\h}^{2} &= \limsup_{j\rightarrow \infty} e^{(N-2s)\theta_{j}} [u_{j}]_{\h}^{2}\nonumber \\
&= \limsup_{j\rightarrow \infty} \Bigl[ e^{(N-2s)\theta_{j}} \iint_{\R^{2N}} \frac{u_{j}(x) -u_{j}(y)}{|x-y|^{N+2s}} (u(x)-u(y)) \, dxdy \nonumber \\
&+ q \sum_{i=1}^{k} e^{(\alpha_{i}+\beta_{i})\theta_{j}} \langle \erre_{i}'(e^{\beta_{i}\theta_{j}}u_{j}), u-u_{j}\rangle \nonumber \\
&+ e^{N\theta_{j}}\int_{\R^{N}} g_{2}(u_{j})(u-u_{j})  \, dx- e^{N\theta_{j}} \int_{\R^{N}} g_{1}(u_{j})(u-u_{j})  \, dx\Bigr]\nonumber \\
&\leq [u]_{\h}^{2}. 
\end{align}
Therefore (\ref{21}) and (\ref{26}) give
\begin{equation}\label{27}
\lim_{j\rightarrow \infty} [u_{j}]_{\h}^{2}= [u]_{\h}^{2},
\end{equation}
which, in view of (\ref{22}), yields
\begin{equation}\label{28}
\lim_{j\rightarrow \infty} \int_{\R^{N}} g_{2}(u_{j})u_{j}  \, dx = \int_{\R^{N}} g_{2}(u)u  \, dx. 
\end{equation}
Since $g_{2}(t)t= mt^{2}+h(t)$, with $h$ a positive and continuous function, by Fatou's Lemma follows that
\begin{align}
&\int_{\R^{N}} h(u)  \, dx \leq \liminf_{j\rightarrow \infty} \int_{\R^{N}} h(u_{j}) \, dx  \label{t1} \\
&\int_{\R^{N}} u^{2}  \, dx \leq \liminf_{j\rightarrow \infty} \int_{\R^{N}} u_{j}^{2} \, dx. \label{t2}
\end{align}
By using (\ref{28}) and (\ref{t1}) we can see that
\begin{align*}
\limsup_{j\rightarrow \infty} \int_{\R^{N}} mu_{j}^{2} \, dx 
&= \limsup_{j\rightarrow \infty} \int_{\R^{N}} (g_{2}(u_{j})u_{j}- h(u_{j})) \, dx \\
&=\int_{\R^{N}} g_{2}(u)u \, dx+  \limsup_{j\rightarrow \infty} \left(-\int_{\R^{N}} h(u_{j}) \, dx\right) \\
&= \int_{\R^{N}} (mu^{2}+ h(u)) \, dx -\liminf_{j\rightarrow \infty} \int_{\R^{N}} h(u_{j})\, dx\\
&= \int_{\R^{N}} mu^{2} \, dx+\int_{\R^{N}} h(u) \, dx -\liminf_{j\rightarrow \infty} \int_{\R^{N}} h(u_{j})\, dx\\
&\leq \int_{\R^{N}} mu^{2} \, dx
\end{align*}
which, together with (\ref{t2}), implies that $u_{j}\rightarrow u$ strongly in $L^{2}(\R^{N})$. Then, we have proved that $u_{j}\rightarrow u$ strongly in $H^{s}_{\rm rad}(\R^{N})$. Since $b_{n}>0$, $u$ is a nontrivial critical point of $\F_{q}$ at the level $b_{n}$. 

\end{proof}

\noindent
Now, we are ready to prove the main result of this Section:
\begin{proof}[Proof of Theorem \ref{thm1.1}]
Let $h\geq 1$. Since $b_{n}\rightarrow +\infty$ (see {\it (2)} of Lemma \ref{lem2.2}), up to a subsequence, we can consider $b_{1}<b_{2}<\dots<b_{h}$. Then, in view of Lemma \ref{lem2.6}, we define $q(h)=q_{h}>0$ and we get the thesis. 

\end{proof}

\section{Proof of Theorem \ref{thm1.2}}
\noindent
In this Section we give the proof of Theorem \ref{thm1.2}. Let us introduce the following functional
\begin{equation*}
\F_{q}(u)= \frac{1}{2} \left( p+\frac{q}{2} (1-s)[u]_{\h}^{2} \right) [u]_{\h}^{2} -\int_{\R^{N}} G(u) \, dx. 
\end{equation*}
In view of Theorem \ref{thm1.1}, it is enough to verify that 
\begin{equation}\label{R}
\erre(u)=\frac{1-s}{4}[u]_{\h}^4
\end{equation}
satisfies the assumptions ($\erre_1$)-($\erre_5$). \\
Clearly $\erre$ is an even and nonnegative $\C^{1}$-functional in $\h$. Since $[u]^{2}_{\h}\leq \|u\|^{2}_{\h}$, we can see that the assumptions ($\erre_1$) and ($\erre_2$) are satisfied. \\
Regarding ($\erre_3$), suppose that $u_{j}\rightharpoonup u$ weakly in $H^{s}_{\rm rad}(\R^{N})$ and $[u_{j}]_{H^{s}(\R^{N})}^{2}\rightarrow \ell\geq 0$. If $\ell=0$, then we have finished. Let us assume $\ell>0$. From the weak lower semicontinuity, we have
\begin{equation}\label{liminf}
[u]_{\h}^{2} \leq \liminf_{j\rightarrow \infty} [u_{j}]_{\h}^{2}.
\end{equation}
By using the following properties of $\liminf$ and $\limsup$ for sequences of real numbers
\begin{align*}
&\limsup_{j\rightarrow \infty} \, a_{j} b_{j} =a \limsup_{j\rightarrow \infty} \,b_{j} \,\, \mbox{ if } a_{j}\rightarrow a>0, \\
&\limsup_{j\rightarrow \infty} \, (a_{j}+ b_{j})=a+ \limsup_{j\rightarrow \infty} \, b_{j} \,\, \mbox{ if } a_{j}\rightarrow a, \\
&\limsup_{j\rightarrow \infty} \, (-a_{j}) = - \liminf_{j\rightarrow \infty} \, a_{j}, 
\end{align*}
and by applying (\ref{liminf}), we obtain
\begin{align*}
&\limsup_{j\rightarrow \infty} \,\langle \erre'(u_{j}), u-u_{j} \rangle=\\
&=(1-s)\limsup_{j\rightarrow \infty} \left([u_{j}]_{\h}^{2} \iint_{\R^{2N}} \frac{(u_{j}(x)-u_{j}(y))}{|x-y|^{N+2s}} [(u(x) - u(y))-(u_{j}(x)-u_{j}(y))] \, dx\,dy\right)\\
&= (1-s)\ell \limsup_{j\rightarrow \infty} \iint_{\R^{2N}} \frac{(u_{j}(x)-u_{j}(y))}{|x-y|^{N+2s}} [(u(x) - u(y))-(u_{j}(x)-u_{j}(y))]  \, dx\,dy\\
&=(1-s)\ell \left( \lim_{j\rightarrow \infty} \iint_{\R^{2N}} \frac{(u_{j}(x)-u_{j}(y))}{|x-y|^{N+2s}} (u(x)-u(y)) \, dx\,dy -   \liminf_{j\rightarrow \infty} [u_{j}]_{\h}^{2} \right)\\
&= (1-s)\ell \left( [u]_{\h}^{2} - \liminf_{j\rightarrow \infty} [u_{j}]_{\h}^{2} \right) \leq 0,
\end{align*}
which gives ($\erre_3$). \\
Now, recalling the definition of $u_{t}$ and by using (\ref{R}), it follows that ($\erre_4$) is verified because of
\begin{align*}
\erre(u_{t})&= \frac{1-s}{4} \left( \iint_{\R^{2N}} \frac{|u(\frac{x}{t})-u(\frac{y}{t})|^{2}}{|x-y|^{N+2s}} \, dx \, dy \right)^{2}\\
&=\frac{(1-s)\,t^{2(N-2s)}}{4}  \left(\iint_{\R^{2N}} \frac{|u(x)-u(y)|^{2}}{|x-y|^{N+2s}} \, dx \, dy \right)^{4}\\
&= t^{2(N-2s)}\erre(u).
\end{align*}
Finally, we prove the condition ($\erre_5$). By using a change of variable, we can see that, for any $g\in \mathcal{O}(N)$  
\begin{equation*}
\erre(u(g \cdot))= \frac{1-s}{4}[u(g \cdot)]_{\h}^{4}= \frac{1-s}{4} [u]_{\h}^{4}=\erre(u). 
\end{equation*}
Then, by applying Theorem \ref{thm1.1}, we can infer that for any $h\in \N$, there exists $q(h)>0$ such that for any $0<q<q(h)$ the functional $\F_{q}$ admits at least $h$ couples of critical points in $H^{s}(\R^{N})$ with radial symmetry. This means that (\ref{K}) admits at least $h$ couples of weak solutions in $H^{s}_{\rm rad}(\R^{N})$.

\end{document}